\documentclass[12pt]{article}

\usepackage{mathrsfs}
\usepackage{amsmath}
\usepackage{amsfonts}
\usepackage{amssymb}

\usepackage{amsmath, amsthm, amsfonts, amssymb}
\newcommand{\CC}{{\mathbb C}}

\def \-{\bar}

\newtheorem{theorem}{Theorem}[section]
\newtheorem{lemma}[theorem]{Lemma}
\newtheorem{corollary}[theorem]{Corollary}

\date{}

\begin{document}

\title{\bf Submanifolds of some Hartogs domain and the complex Euclidean space
\footnote{Keywords: Hartogs domain, isometric embedding, Bergman metric, Nash algebraic function, rational function; 2010 Subject classes: 32H02, 32Q40, 53B35.}}

\author{
\ \ Xu Zhang, \ \ Donghai Ji %\footnote{ Supported in part by National Science Foundation grant}
}

\vspace{3cm} \maketitle

\begin{abstract}
Two K\"ahler manifolds are called relatives if they admit
a common K\"ahler submanifold with the same induced metrics.
In this paper, %we %provide  a sufficient condition to determine
%whether a real analytic K\"ahler manifold is not a relative to
%a complex space form equipped with its canonical metric.
%As an application, 
we show that a Hartogs domain over an irreducible bounded symmetric domain equipped with the Bergman metric is not a relative
to the complex Euclidean space. This generalizes the results in \cite{[CN], [CH]} and the novelty here is that the Bergman kernel of the Hartogs domain is not necessarily Nash algebraic.
\end{abstract}

\bigskip
\section{Introduction}

Holomorphic isometric embedding is an important topic in complex geometry.
%have been studied extensively by many authors. 
Calabi \cite{[C]} obtained the celebrated results on 
the global extendability and rigidity  of a local holomorphic
isometry into a complex space form. %among many other important
%results.
 In particular, Calabi proved that any complex space form cannot
be locally isometrically embedded into another complex space form
with a different curvature sign with respect to the canonical
K\"ahler metrics, respectively. %In his paper, Calabi introduced the
%so called diastasis function and reduced the metric tensor equation to the functional identity for the diastasis functions. 
Along this line, Di Scala and Loi generalized Calabi's
non-embeddability  result to the case of Hermitian symmetric spaces
of different types \cite{[DL1]}. For related problems on the existence of the K\"ahler-Einstein submanifold or the K\"ahler-Ricci soliton of complex space forms, the interested readers may refer to \cite{[U1], [LM2]}. For the study of holomorphic isometric embeddings between bounded symmetric domains, the interested readers may refer to \cite{[Mo1], [Mo3], [Y1]} and references therein.

On the other hand, Umehara \cite{[U2]} studied a more general question
whether two complex space forms can share a common submanifold with
the induced metrics and proved that
two complex space forms with different curvature signs cannot share
a common K\"ahler submanifold by using Calabi's idea of diatasis. When two complex manifolds share a
common K\"ahler submanifolds with induced metrics,  they are called relatives by Di Scala and Loi
 \cite{[DL2]}. Furthermore, Di Scala and Loi proved
that a bounded domain equipped with its canonical Bergman metric can not be
a relative to a Hermitian symmetric space of compact type equipped
with the canonical metric. Since any irreducible Hermitian
symmetric space of compact type can be holomorphically isometrically
embedded into a complex project space by the classical
Nakagawa-Takagi embedding, in order to show that a
K\"ahler manifold is not a relative of a projective manifold with the
induced metric, it suffices to show that it is not a relative to the
complex projective space with the Fubini-Study metric. In fact, it
follows from the result of Umehara \cite{[U2]}, the complex Euclidean
space and the irreducible Hermitian symmetric space of compact type
cannot be relatives. Later, Huang and Yuan showed that 
 a complex Euclidean space and a Hermitian symmetric space of
noncompact type cannot be relatives \cite{[HY2]}. Cheng, Di Scala and Yuan also considered the relativity problem for indefinite complex space forms \cite{[CDY]}. 

More recently, Su, Tang and Tu showed that a complex Euclidean space and the symmetrized polydisc are not relatives \cite{[STT]}. Cheng and Hao provided a sufficient condition to determine whether a real analytic K\"ahler manifold is not a relative to a complex space form equipped with its canonical metric. As consequences, minimal domains, bounded homogeneous domains and some Hartogs domains equipped with their Bergman metrics are not relatives to the complex Euclidean space \cite{[CH]}. In these results, one key property for the canonical Bergman metrics is that the Bergman kernel functions are Nash algebraic (even rational) functions. 

In \cite{[CN]}, Cheng and Niu considered the relativity problem for the complex Euclidean space  and the Cartan-Hartogs domain. The fundamental difference in this problem is that the Bergman kernel function of the Cartan-Hartogs domain is no longer Nash algebraic. Cheng and Niu were able to show a special case that any submanifold passing the origin of the Cartan-Hartogs domain cannot be the submanifold of  the complex Euclidean space. In this paper, we solve the general case, even for the Hartogs domain over a bounded, complete circular, homogeneous, Lu Qi-Keng domain equipped with the Bergman metric.

Let $D\subset \mathbb{C}^{d}$ be a domain and $\varphi$ be a continuous positive function on $D$.
The domain
\begin{equation}\label{equ:Hartogs domain}
\Omega=\left\{(\xi, z)\in \mathbb{C}^{d_{0}}\times D : ||\xi||^{2}<\varphi(z)\right\}
\end{equation}
is called a Hartogs domain over  $D$ with $d_{0}$-dimensional fibers.
%
%In \cite{[Li]}, Ligocka gave a series representation formula of the
%Bergman kernel of the Hartogs domain involving weighted Bergman kernels of the base domain.
%She proved  that the Bergman kernel of $\Omega$ is
%\begin{equation}\label{equ:Ligocka formula2}
%K((z,\xi),\overline{(w,\zeta)})=\sum_{k=0}^{\infty}\frac{(k+1)_{d_{0}}}{\pi^{d_{0}}}K_{D,\varphi^{k+d_{0}}}(z,\overline{w})<\xi,\zeta>^{k},
%\end{equation}
%where $K_{D,\varphi^{k+d_{0}}}$ stands for the weighted Bergman kernel  with respect to the weight $\varphi^{k+d_{0}}$,
%$<\cdot , \cdot> $ denotes the scalar product in $\mathbb{C}^{d_{0}}$,
%$(k+1)_{d_{0}}$ denotes the Pochhammer polynomial of degree $d_{0}$, i.e.
%$(k+1)_{d_{0}}=\frac{\Gamma(k+1+d_{0})}{\Gamma (k+1)}.$
When  $D$ is a bounded homogeneous domain  and $\varphi(z)=K_D(z,\overline{z})^{-s}$,
Ishi, Park and Yamamori \cite{[IPY]} proved that the Bergman kernel is
\begin{equation}\label{B kernel}
K((z,\xi),\overline{(w, \eta)})
=\frac{K_{D}(z,\overline{w})^{d_{0}s+1}}{\pi^{d_{0}}}
\sum_{j=0}^{d}\frac{c(s,j)(j+d_{0})!}{(1-t)^{j+d_{0}+1}},
\end{equation}
with $t=K_{D}(z,\overline{w})^{s}<\xi, \eta>$,
where the constants $c(s,j)$ satisfy that
$F(ks)=\sum_{j=0}^{d} c(s, j)(k+1)_{j}$
and $F$ is the polynomial given by (18) in \cite{[IPY]}.
If $D$ is a bounded homogeneous domain whose Bergman kernel function
$K_{D}(z, \overline{w})$ is a rational function on $D\times\hbox{conj}(D)$,
then $K((z,\xi),\overline{(w,\eta)})$ is also a rational function
on $\Omega\times\text{conj}(\Omega)$ for $s \in \mathbb{Z}$. 
In this case, $\Omega$ equipped with the Bergman metric is not a relative to $\mathbb{C}^{n}$ equipped with the Euclidean metric (cf. \cite{[CH]}).
However, in general, for $s \not\in \mathbb{Z}$, $K((z,\xi),\overline{(w,\eta)})$ is a product of a rational function to a power and a rational function in $t$ and the problem becomes totally different.

Denote the Euclidean metric on $\CC^n$ by  $\omega_{\CC^n}$ and the Bergman metric on $\Omega$ by $\omega_\Omega$. %For each $1 \leq j \leq J$, let the bounded symmetric domain $\Omega_j \subset \CC^{m_j}$ be the Harish-Chandra realization of an irreducible Hermitian symmetric space of noncompact type and let $\omega_{\Omega_j}$ be the Bergman metric on $\Omega_j$. Let $U \subset \CC^\kappa$ be a connected open set and $\omega_U$ be a K\"ahler metric on $U$, not necessarily complete. %In \cite{DL2}, Di Scala and Loi raised the following interesting question: whether?
We will show that  there do not simultaneously exist  holomorphic isometric immersions
  $F=(f_1, \cdots, f_n): (U, \omega_U) \rightarrow (\CC^n, \omega_{\CC^n})$ and $G=(g_1, \cdots, g_{d_0}, h_1, \cdots, h_d): (U, \omega_U) \rightarrow (\Omega, \omega_\Omega)$. %with $\mu_1, \cdots, \mu_J$ positive real numbers.
   %such that $\omega_D =  F^*\omega_{\CC^n} = G^*\omega_\Omega$.
 As a consequence,
  a complex Euclidean space and a Hartogs domain over an irreducible bounded symmetric domain
  %, complete circular, homogeneous, Lu Qi-Keng domain 
  equipped with the Bergman metric cannot be relatives.
 % If $(\CC^n, \omega_{\CC^n})$ and $(\Omega, \omega_\Omega)$ are
 %  relatives,
  %  by restricting to an open  piece of a complex line, one can just assume $D \subset \CC$ to get a contradicton.
Indeed, we prove the following slightly stronger result:

\begin{theorem} \label{main1}
Let $\Omega$ be a Hartogs domain over an irreducible bounded symmetric domain $D$ defined by (\ref{equ:Hartogs domain}).
Let $U \subset \CC$ be a connected open subset. %equipped with a Hermitian metric $\omega_D$. %equipped with a K\"ahler metric $\omega_D$.
%Then there do not simultaneously  exist non-constant holomorphic maps
Suppose that
% $F: (D, \omega_D) \rightarrow (\CC^n, \omega_{\CC^n})$ and $G: (D, \omega_D) \rightarrow (\Omega, \omega_\Omega)$
$F: U \rightarrow \CC^n$ and $%(g_1, \cdots, g_{d_0}, h_1, \cdots, h_d)
G: U \rightarrow
\Omega$ are holomorphic
mappings such that

\begin{equation}\label{isometry}
F^*\omega_{\CC^n} = \mu G^*\omega_{\Omega} ~~\text{on}~~U
\end{equation}
 for a real  constants $\mu$. Then $F$ must be a constant map. 
 \end{theorem}

By the same argument, we are able to show a slightly general result. 
Assume each $\Omega$ to be a Hartogs domain over an irreducible bounded symmetric domain defined by (\ref{equ:Hartogs domain}).
Suppose $F: U \rightarrow \CC^n$ and $%(g_1, \cdots, g_{d_0}, h_1, \cdots, h_d)
G_j: U \rightarrow
\Omega_j$ are holomorphic
mappings such that

\begin{equation}\label{isometry}
F^*\omega_{\CC^n} =\sum_{j=1}^J \mu_j G_j^*\omega_{\Omega_j} ~~\text{on}~~U
\end{equation}
 for certain real  constants $\mu_1,
\cdots, \mu_J$. Then $F$ must be a constant map. Furthermore, if all
$\mu_j's$ are positive, then  $G$ is also a constant map.

The next result follows directly from Theorem \ref{main1}.

\begin{corollary}\label{corollary}
There does not exist a K\"ahler manifold $(X, \omega_X)$ that can be
holomorphic isometrically embedded into the complex Euclidean space
 $(\CC^n, \omega_{\CC^n})$ and  also  into  a Hartogs domain $(\Omega, \omega_\Omega)$ over an irreducible bounded symmetric domain defined by (\ref{equ:Hartogs domain}).
\end{corollary}

The corollary generalizes the results in \cite{[CH], [CN]} for arbitrary $s \in \mathbb{R}$. Again, the key difference in this work from previous works is the fact that the Bergman kernel for a Hartogs domain over an irreducible bounded symmetric domain is no longer Nash algebraic, which makes the general argument in \cite{[HY2]} inapplicable (cf. \cite{[CH]}). 

We conclude the paper with a remark. One notes that all results stated above hold for the Hartogs domain $\Omega$ over a bounded, complete circular, homogeneous, Lu Qi-Keng domain equipped with the Bergman metric. All the argument works the same provided that Lemma \ref{auto} holds for such Hartogs domain and we explain there why this is the case.

%We notice that  in our theorem, a normalization of the metric up to
%a constant does not affect the validity of the result. However, as
%pointed out in \cite{DL2}, even $(\CC\PP^n, \omega_{FS})$ and
%$(\CC\PP^n, \gamma \omega_{FS})$ may not be relatives if $\gamma$ is
%an irrational number (see also \cite{C}\cite{HY}).

\bigskip

%{\bf Acknowledgement}: We thank  Professor Xiaoliang Cheng and Yuan Yuan for helpful discussions.

\bigskip

\bigskip

\section{Proof of Theorem \ref{main1}}

Our  proof fundamentally uses  ideas developed in \cite{[HY1], [HY2]}. %The idea of the proof comes from in \cite{[HY2]}. 
First, we prove a generalization of Lemma 2.2 in  \cite{[HY2]}.

  \begin{lemma}\label{cont}
  Let $V \subset \CC^k$ be a connected open set. Let $H_1(\xi), \cdots,  H_K(\xi)$
    and $H(\xi)$  be  holomorphic Nash algebraic functions on $V$ and $R(t)$ be a holomorphic rational function.
    Assume that
    
 \begin{equation}\label{form}
   \exp^{H(\xi)} =  R(H_1^{\mu_1}(\xi) H_2^{\mu_2}(\xi)) \prod_{j=3}^K  H_j^{\mu_j}(\xi), \ \xi\in V,
   \end{equation}\label{trans}
    for certain    real numbers $\mu_1, \cdots, \mu_K$. Then
    $H(\xi)$ is constant.
   \end{lemma}

   \begin{proof} We include a brief argument here and refer the readers to  the proof of Lemma 2.2 in  \cite{[HY2]} for details.   
  
   Suppose that $H$ is not constant. After a linear transformation in $\xi$ and the dimension reduction, we can assume, without loss of generality,
   that,  $H(\xi)$ is not  constant for a certain fixed $\xi_2, \cdots, \xi_k$.
   Then $H$ is a non-constant  Nash-algebraic holomorphic   function in $\xi_1$ for such fixed  $\xi_2, \cdots, \xi_k$.
   Hence, we can assume that $j=1$ to achieve a contradiction.
 Write $H=H(\xi)$ and $H_j=H_j(\xi)$ for each $1 \leq j \leq K$. 
 Use  $S \subset \CC$ to denote the union  of branch points, poles and zeros  of
     $H(\xi)$ and
     $H_j(\xi)$ for each $j$.  Given a  $p \in \CC \setminus S$ and a real curve in $\CC \setminus S$ connecting $p$ and $V$,
     by holomorphic continuation, the equation (\ref{trans}) holds on an open neighborhood of the curve.
 Suppose that  some branch $H^{(*)}$ of $H$ blows up
at $\xi_0 \in \mathbb{C} \cup \{\infty\}$. Then by Puiseux expansion of $H$ near $\xi_0$, $|e^{H^{(*)}(x)}|$  goes to infinity exponentially fast as $\xi$ tends to $\xi_0$. However, the right hand side of (\ref{trans}) grows at most polynomially. This is a contradiction.
 \end{proof}

The following key result is a consequence of Proposition 2 in \cite{[WYZR]}. Although in \cite{[WYZR]} a more general result was merely stated for the Hartogs domain over an irreducible bounded symmetric domain, if the reader applies the argument in \cite{[YLR], [ABP]} carefully, it is not difficulty to find out Lemma \ref{auto} also holds for the Hartogs domain $\Omega$ over a bounded, complete circular, homogeneous, Lu Qi-Keng domain.

\begin{lemma}\label{auto}
For any point $(\xi_0, z_0) \in \Omega$, there exists $\Phi\in$ Aut$(\Omega)$, such that $\Phi (\xi_0, z_0) = (\xi'_0, 0)$.
\end{lemma}
 
 Now, we apply the argument in \cite{[LM2]} to show the algebraicity. %and the original idea of the argument is from \cite{[HY2]}.

\begin{lemma}\label{algebra}
Writing $S= \{\phi_1, \cdots, \phi_l \}= \{f_1, \cdots, f_n, g_1, \cdots, g_{d_0}, h_1, \cdots, h_d \}$, then there exists a maximal algebraic independent subset $\{\phi_1, \cdots, \phi_k\} \subset S$  over the field $\cal{R}$ of rational functions on $U$, and Nash algebraic functions $\hat\phi_j(t, X_1, \cdots, X_k)$,     such that $\phi_j(t) = \hat\phi_j(t, \phi_1(t), \cdots, \phi_k(t))$ for all $1 \leq j \leq l$ after
shrinking $U$ toward the origin if needed.
\end{lemma}

\begin{proof}
The proof is identical to the argument in \cite{[LM2]} and we include the proof here for completeness. Let $k$ be the transcendence degree of the field extension $\cal{R}(S) / \cal{R}$. If $k=0$, then every element in $S$ is holomorphic Nash algebraic. Otherwise, we have a maximal algebraic independent subset over $\cal R$, denoted by $\phi_1, \cdots, \phi_k$. By the algebraic version of the existence and uniqueness of the implicit function theorem, there exist a connected open set, still denoted by $U$ and Nash algebraic functions $\hat\phi_j(t, X_1, \cdots, X_k)$, defined in a neighborhood $\hat U$ of $\{(s, \phi_1(s), \cdots, \phi_k(s)) | s\in U \}$,
 such that $\phi_j(t) = \hat\phi_j(t, \phi_1(t), \cdots, \phi_k(t))$ for all $s\in U$ and $1 \leq j \leq l$.
 \end{proof}

We are now in a position to prove the main result.
\medskip

{\bf Proof of Theorem \ref{main1}:}

Without loss of generality, we assume $0 \in U$. By composing $F$ with Aut$(\mathbb{C}^n)$ and $G$ with $\Phi \in$ Aut$(\Omega)$, we may assume $F(0)=0, G(0)=(\xi_0, 0)$. Since the metrics are invariant under automorphisms, the equation (\ref{isometry}) still holds. We still denote the maps by $F$ and $G$.
%Moreover, by slicing $U$, we may assume $U$ to be  a domain in ${\mathbb C}$. 
Let $F=(f_1,
\cdots, f_n): D \rightarrow \CC^n,\ \  G=(g_1, \cdots, g_{d_0}, h_1, \cdots, h_d): U
\rightarrow \Omega$ be holomorphic
maps. %satisfying equation (\ref{isometry}). %Without loss of generality, assume that $0 \in D$ and  $F(0)=0, G(0)=0$. 
We argue by
contradiction. Suppose that $F$ is not constant. %Assume that $F,\ G$ are non-constant holomorphic maps
%such that $ F^*\omega_{\CC^n} = G^*\omega_{\Omega}$.
By equation (\ref{isometry}), we  have

$$\partial\bar\partial \left(\sum_{i=1}^n |f_i(s)|^2 \right)=  \mu \partial\bar\partial \log K(G(s), \overline{G(s)})~~\text{for~~}s \in U,$$
where $K((\xi, z), \overline{(\eta ,w)})$ %= \sum_l h_{jl}(\xi)
%\overline{h_{jl}(\eta)}$ 
is the Bergman kernel on $\Omega$. 
%and $\{h_{jl}(\xi)\}$ is an orthonormal  basis of $L^2$ integrable
%holomorphic functions over $\Omega_j$. Note that 
Since $\Omega$ is a complete circular domain, by the standard argument (cf. \cite{[Mo1], [HY2]}), %in the Harish-Chandra realization. Therefore,
%the Bergman kernel of $\Omega_j$ satisfies the identity $K_j(e^{\sqrt{-1}\theta}\xi, \overline{e^{\sqrt{-1}\theta}\eta}) = K_j(\xi, \overline{\eta})$
% for any $\theta \in \RR$ and any $\xi, \eta \in \Omega_j$. This implies  $K_j(e^{\sqrt{-1}\theta}\xi, 0) = K_j(\xi, 0)$. Therefore $K_j(\xi, 0)$ is a positive
%  constant. In another word, 
  $K((\xi, z), \overline{(\eta ,w)})$ does not contain any nonconstant pure holomorphic terms in $(\xi, z)$.  %After normalization, we can assume tha $K(\xi, 0) =1$.
By comparing the
 pure holomorphic and anti-holomorphic terms in $(\xi, z)$ (cf. \cite{[CU]}), one can get rid of $\partial\bar\partial$ to obtain the following functional identity:

\begin{equation}\label{function}
\sum_{i=1}^n |f_i(s)|^2 = \mu \log K(G(s),
\overline{G(s)}) ~~\text{~for~any~}s \in U.
\end{equation}
%Here $c$ is a certain constant.
%
%An alternative way to obtain (\ref{function}) is from the explicit computation of the Calabi's diastasis function at 0 \cite{[C]}.
%
After polarization, (\ref{function}) is equivalent to
\begin{equation}\label{polarization}
\sum_{i=1}^n f_i(s) \bar{f_i}(t) =\mu \log
K(G(s), \bar{G}(t)) ~~\text{~for~~}(s, t) \in U\times
\hbox{conj}({U}),
\end{equation}
where $\bar{f}_i(s) = \overline{f_i(\overline{s})}$ and
$\hbox{conj}({U})=\{z \in \CC | \bar s \in U\}$.

Define $$\Psi(s, X, t) = \sum_{i=1}^n \hat f_i(s, X) \bar f_i(t)-
 \mu \log K((\hat g_1(s, X), \cdots, \hat h_{d}(s, X)), (\bar g_1(t), \cdots, \bar h_d(t)))$$
and
$$\Psi_i (s, X, t)= \frac{\partial^i \Psi }{\partial t^i} (s, X, t)$$ for $(s, X, t) \in \hat U \times \hbox{conj}(U)$,
where $X=(X_1, \cdots, X_k)$.

\begin{lemma}\label{quantity}
For any $t$ near 0 and any $(s, X)\in \hat U$, $\Psi(s, X, t) \equiv \Psi(s, X, 0)$.
\end{lemma}
\begin{proof}
It suffices to show that $\Psi_i (s, X, 0) \equiv 0$ for all $i >0$.
Since $K_D(z, \bar w)$ is a rational function and also $K_D((\hat h_1(s, X), \cdots, \hat h_d(s, X)), \overline{(h_1(0), \cdots h_d(0))} ) = K_D((\hat h_1(s, X), \cdots, \hat h_d(s, X)), (0, \cdots, 0))  \equiv 1$, then  $\Psi_i (s, X, 0)$ is a Nash algebraic function
 in $(s, X)$. Assume it is not constant.  Then there exists a holomorphic polynomial $P(s, X, y)=A_d(s, X)y^d + \cdots + A_0(s, X)$ of degree $d$ in $y$, with $A_0(s, X) \not\equiv0$ such that $P(s, X, \Psi_i(s, X, 0)) \equiv 0$.
As $\Psi(s, \phi_1(s), \cdots, \phi_k(s), t) \equiv 0$ for $s\in {U}$, it
follows that $\Psi_i(s, \phi_1(s), \cdots, \phi_k(s), 0) \equiv 0$ and
therefore $A_0(s, \phi_1(s), \cdots, \phi_k(s)) \equiv 0$. This means that
$\{\phi_1(s), \cdots, \phi_k(s)$\} are algebraic dependent over
$\cal{R}$. This is a contradiction and it follows that $\Psi_i (s, X, 0)$ is a constant. 
Therefore, $\Psi_i (s, X, 0) = \Psi_i(s, \phi_1(s), \cdots, \phi_k(s), 0) \equiv 0.$

%Since $\Psi(z, X, w)$ is holomorphic in $w$ and $\Psi(z, X, 0) \equiv 0$, then $\Psi(z, X, w) \equiv 0$.
\end{proof}

\begin{lemma}\label{nonconstant}

There exists $(s_0, t_0) \in U \times \hbox{conj}(U)$ with $t_0\not=0$ such that
\begin{equation}\notag
%K(X, \hat g_{l+1}(z_0, X), \cdots, \hat g_{m}(z_0, X), g_1(w_0), \cdots, g_m(w_0))
\sum_{i=1}^n \hat f_i(s_0, X) \bar f_i(t_0) \not\equiv \text{constant}.\end{equation}
\end{lemma}

\begin{proof}
Assume the conclusion is false. Letting $t=\overline{s}$ and $X=(\phi_1(s), \cdots, \phi_k(s))$, %in equation (\ref{zero}), the right hand side is identity 1. I
it follows that

$$\sum_{i=1}^n |f_i(s)|^2 = \sum_{i=1}^n \hat f_i(s, \phi_1(s), \cdots, \phi_k(s)) \bar f_i(\overline{s}) \equiv \text{constant},\ \hbox{over}\ U.$$ This implies that $f_i(s) \equiv 0$ for all $1 \leq i \leq n$ and therefore contradicts to the assumption that $F=(f_1, \cdots, f_n)$ is a non-constant map. 
\end{proof}

%Assume that the minimal  polynomial of $H$ is given by $p(\xi, X)=
%A_d(\xi)X^d + \cdots + A_0(\xi)$ such that $p(\xi, H(\xi)) \equiv
%0$. Denote the branches of $H$ by $\{H^{(1)}, \cdots, H^{(d)}\}$ and
%these branches can be obtained through $H$ by holomorphic
%continuation.  Denote the corresponding branches for $H_k$ obtained
%by holomorphic continuation by $\{H_k^{(1)}, \cdots, H_k^{(d)}\}$.
%Let $\xi_0$ be a zero of $\frac{A_d}{A_0}$ or $\xi_0=\infty$ if
%$\frac{A_d}{A_0}$ is a constant. Then some branches of $H$ blow up
%at $\xi_0$. Without loss of generality, assume that $\xi_0=\infty$.
%Assume that (\ref{trans}) holds in a neighborhood of $\infty$ after
%holomorphic continuation from the original equality. By the Puiseux
%expansion, we can assume that $$H(\xi)=\sum_{\beta=\beta_0,
%\beta_0-1,\cdots,
%-\infty}a_{\beta}\xi^{\beta/N_0}=a_{\beta_0}\xi^{\beta_0/N_0}+o(|\xi|^{\beta_0/N_0})$$
%for $|\xi|>>1$ with $a_{\beta_0}\not = 0$ and $\beta_0, N_0 >0$.
%Without loss of generality, we assume that  $a_{\beta_0}>0$. Now,
%when $\xi\ra \infty $ along the positive $x$-axis, for the branch
%$H^{(*)}$, which corresponds to $\xi^{\beta_0/N_0}$ taking positive
%value along this ray in its Puiseux expansion, we have
 %$|e^{H^{(*)}(x)}|\ge e^{\left(\frac{a_{\beta_0}}{2}x^{\frac{\beta_0}{N_0}}\right)}$ %with $C_0>0$
 %as $x\ra +\infty$. However, the
 %right hand side of (\ref{trans}) grows at most
 %polynomially. This is a contradiction.

\bigskip

%Now, assume that $l>0$.  By re-ordering the lower index, let
%$\mathcal{G}= \{g_1(z), \cdots, g_l(z)\}$ be the maximal algebraic
%  independent subset in $\mathfrak{F}$. It follows that the transcendence degree of $\mathfrak{F} / \mathfrak{R}(\mathcal{G})$ is 0.
%  Then   there exists a small connected open subset $U$ with $0\in
%  \overline{U}$ such that for each $j_\a$ with $g_{j_\a}\not \in \mathcal{G}$,
%   we have a holomorphic Nash algebraic function $\hat g_{j_ \alpha}(z, X_1, \cdots, X_l) $ in the neighborhood
%   $\hat U$ of $\{(z, g_1(z), \cdots, g_l(z)) | z\in U\}$ in $\CC\times\CC^l$ such
%   that it holds that
%    $g_{j_ \alpha}(z) = \hat g_{j \alpha} (z, g_1(z), \cdots, g_l(z))$ for any $z \in U$.
 %   Then by Lemma \ref{algebraicity}, for each $1\leq i \leq n$, there exists  a holomorphic Nash algebraic function $\hat f_i(z, X_1, \cdots, X_l)$ in $\hat U$
 %  such that $f_i(z) = \hat f_i(z, g_1(z), \cdots, g_l(z))$ for  $z \in U$. %For consistence of notation, also write $g_l (z) =\ $

%It follows from Lemma \ref{quantity} and 

Choosing $s_0, t_0$ as in Lemma  \ref{nonconstant}, %for equation (\ref{zero}), %one has:
%$$\sum_{i=1}^n \hat f_i(z_0, X) \bar f_i(w_0) = \log K_j(X, \hat g_{l+1}(z_0, X), \cdots, \hat g_{m}(z_0, X), g_1(w_0), \cdots, g_m(w_0)),$$ where both sides are nonconstant Nash algebraic functions in $X$.
$\sum_{i=1}^n \hat f_i(s_0, X) \bar f_i(t_0)$ is a nonconstant
holomorphic Nash algebraic function in $X$. It follows from Lemma \ref{quantity} that 
$$e^{\frac{1}{\mu} \sum_{i=1}^n \hat f_i(s_0, X) \bar f_i(t_0)} = \frac{K((\hat g_1(s_0, X), \cdots, \hat h_{d}(s_0, X)), (\bar g_1(t_0), \cdots, \bar h_d(t_0)))}{ K((\hat g_1(s_0, X), \cdots, \hat h_{d}(s_0, X)), (0, \cdots, 0, \bar h_1(0), \cdots, \bar h_d(0)))}.$$
Note that this is precisely in the form of the equation (\ref{form}).
Hence we arrive at a contradiction by Lemma \ref{cont}. Thus
$F$ must be a constant map. This is again a contradiction and it completes the proof of Theorem \ref{main1}.
% once again together with equation (\ref{zero}).
%%%%%%%%%%%%%%%%%%%%%%%%%%%%%%%%%%%%%%%%%%

%%%%%%%%%%%%%%%%%%%%%%%%%%%%%%%%%%%%%%%%%%%%%%%%%%%%%%%%%%%%%%%%

\end{document}